\documentclass[letter]{amsart}

\usepackage{amsthm,amsmath,amssymb}

\newtheorem{theorem}{Theorem}
\newtheorem{lemma}[theorem]{Lemma}
\newtheorem{proposition}[theorem]{Proposition}

\newtheorem{corollary}[theorem]{Corollary}

\theoremstyle{definition}
\newtheorem{definition}[theorem]{Definition}

\begin{document}

\title[Proof of Becker's conjecture]{Becker's conjecture on Mahler functions}


\author{Jason P. Bell}
\address{Department of Pure Mathematics, University of Waterloo, Canada}
\email{jpbell@uwaterloo.ca}

\author{Fr\'ed\'eric Chyzak}
\address{INRIA, Universit\'e Paris--Saclay, France}
\email{Frederic.Chyzak,Philippe.Dumas@inria.fr}

\author{Michael Coons}
\address{School of Mathematical and Physical Sciences\\
University of Newcastle\\
Australia}
\email{Michael.Coons@newcastle.edu.au}

\author{Philippe Dumas}

\date{\today}
\thanks{The research of J.~P.~Bell was partially supported by an NSERC Discovery Grant. M.~Coons was visiting the Alfr\'ed R\'enyi Institute of the Hungarian Academy of Sciences during the time this research was undertaken; he thanks the Institute and its members for their kindness and support.}

\begin{abstract} In 1994, Becker conjectured that if $F(z)$ is a $k$-regular power series, then there exists a $k$-regular rational function $R(z)$ such that $F(z)/R(z)$ satisfies a Mahler-type functional equation with polynomial coefficients where the initial coefficient satisfies $a_0(z)=1$. In this paper, we prove Becker's conjecture in the best-possible form; we show that the rational function $R(z)$ can be taken to be a polynomial $z^\gamma Q(z)$ for some explicit non-negative integer $\gamma$ and such that $1/Q(z)$ is $k$-regular.
\end{abstract}

\maketitle

\section{Introduction}

Let $k\geqslant 2$ be an integer. A Laurent power series $F(z)\in\mathbb{C}((z))$ is called {\em $k$-Mahler} provided there exist a positive integer $d$ and polynomials $a_0(z),\ldots,a_d(z)\in\mathbb{C}[z]$ with $a_0(z)a_d(z)\neq 0$ such that $F(z)$ satisfies the Mahler-type functional equation \begin{equation}\label{MFE} a_0(z)F(z)+a_1(z)F(z^k)+\cdots+a_d(z)F(z^{k^d})=0.\end{equation} The minimal $d$ for which such an equation exists is called the {\em degree} of $F(z)$.

There has been a flurry of recent activity involving the study of Mahler series---see, e.g. \cite{AF2017,BC2017,BCZ2016,BV2016,C2016,GT2018,P2015,R2017,R2018}---in large part due to the fact that one can often deduce transcendence of special values of Mahler series by knowing transcendence of the series itself, and also due to the guiding principle that much of the theory of Mahler series should mirror the much better developed theory of solutions to homogeneous differential equations.

A special subclass of Mahler functions is the ring of $k$-regular power series. These functions are defined from their coefficient sequences. More specifically, a power series $F(z)=\sum_{n\geqslant 0} f(n)z^n$ is {\em $k$-regular} provided there is a positive integer $D$, vectors ${\boldsymbol\ell},{\bf c}\in\mathbb{C}^{D\times 1}$, and matrices ${\bf A}_0,\ldots,{\bf A}_{k-1}\in\mathbb{C}^{D\times D}$, such that for all $n\geqslant 0$, $$f(n)={\boldsymbol\ell}^T {\bf A}_{i_s}\cdots {\bf A}_{i_0}{\bf c},$$ where $(n)_k=i_s\cdots i_0$ is the base-$k$ expansion of $n$. Allouche and Shallit \cite{AS1992} introduced $k$-regular sequences in the early nineties as a generalisation of $k$-automatic sequences; that is, sequences obtained from a deterministic finite-state automaton which takes as input the base-$k$ expansion of $n$ and outputs the $n$-th term of the sequence. By refining the proof of the known result that $k$-automatic power series are $k$-Mahler, Allouche \cite[Theorem 1]{A1987} implicitly showed\footnote{The proof of the latter implication is indeed buried in his proof of the former, as a substep that does not rely on the finiteness of the base field \cite[p. 253 and~254]{B1994}.} that a $k$-regular power series is $k$-Mahler. Since all $k$-automatic sequences are $k$-regular sequences, there is a natural hierarchy: $$\{\mbox{$k$-automatic functions}\} \subset \{\mbox{$k$-regular functions}\}\subset \{\mbox{$k$-Mahler functions}\}.$$ Additionally, both inclusions are proper. For example, consider the three paradigmatic examples for $k=2$ of degree one, $$T(z)=\prod_{j\geqslant 0}(1-z^{2^j}),\quad S(z)=\prod_{j\geqslant 0}(1+z^{2^j}+z^{2^{j+1}})\quad\mbox{and}\quad T(z)^{-1}=\prod_{j\geqslant 0}(1-z^{2^j})^{-1}.$$ The function $T(z)$ is the generating power series of the Thue--Morse sequence $t(n)$ over the alphabet $\{-1,1\}$, which is $2$-automatic. The function $S(z)$ is the generating power series of the Stern sequence $s(n+1)$ that counts the number of hyperbinary representations of the number $n+1$, which is $2$-regular but not $2$-automatic. And finally, the function $T(z)^{-1}$ is $2$-Mahler but not $2$-regular; the coefficients $p(n)$ of the power series expansion of $T(z)^{-1}$ count the number of ways of writing the number $n$ as sums of powers of two.

Since a $k$-regular power series is $k$-Mahler, an immediate question arises: {\em can one determine if a solution to \eqref{MFE} is $k$-regular, or not, based solely on properties of the functional equation?} Towards answering this question, Becker \cite[Theorem 2]{B1994} showed that if $a_0(z)=1$, then $F(z)$ is $k$-regular. He conjectured \cite[p.~279]{B1994} that a sort of converse to this result also holds. Specifically, Becker conjectured that {\em if $F(z)$ is a $k$-regular power series, then there exists a nonzero $k$-regular rational function $R(z)$ such that $F(z)/R(z)$ satisfies a Mahler-type functional equation \eqref{MFE} with $a_0(z)=1$.} In view of this conjecture, a power series $F(z)$ is called {\em $k$-Becker} provided it satisfies a functional equation \eqref{MFE} with $a_0(z)=1$.

The historical significance of the $k$-Becker property lies in the fact that zeros of $a_0(z)$ in the minimal Mahler equation \eqref{MFE} for $F(z)$ are values $\alpha$ at which the theorems proving transcendence of $F(\alpha)$ based upon knowledge of algebraic independence of certain related Mahler functions do not apply; this point is highlighted in the works of Loxton and van der Poorten \cite{LV1982,LV1988} and the celebrated result of Nishioka \cite{N1990,N1996}.  In this paper, we prove (a bit more than) Becker's conjecture.

\begin{theorem}\label{main} If $F(z)$ is a $k$-regular power series, there exist a nonzero polynomial $Q(z)$ with $Q(0)=1$ such that $1/Q(z)$ is $k$-regular and a nonnegative integer $\gamma$ such that $F(z)/z^\gamma Q(z)$ satisfies a Mahler-type functional equation \eqref{MFE} with $a_0(z)=1$.
\end{theorem}

\noindent Moreover, if the Mahler-type functional equation of minimal degree for $F(z)$ is known, then the polynomial $Q(z)$ in Theorem \ref{main} can be easily written down. Specifically, if \eqref{MFE} is the minimal functional equation for $F(z)$, and we write $A$ for the set of roots of unity $\zeta$ such that $\zeta^{k^M}\neq\zeta$ for all $M\geqslant 1$ and $a_0(\zeta)=0$, then there is an $N$ depending on $a_0(z)$ such that $$Q(z):=\prod_{\zeta\in A}\prod_{j=0}^{N-1}(1-z^{k^j}\overline{\zeta}^{k^N})^{\nu_\zeta(a_0(z))},$$ where for a given Laurent power series $g(z)$, $\nu_\zeta(g(z))$ is the order of the zero of $g(z)$ at $z=\zeta$. For more details, see the proof of Lemma \ref{lem:norou}. Noting that all of the zeros of the polynomial $Q(z)$ are roots of unity of order not coprime to $k$, we may combine this with a result of Dumas \cite[Th\'eor\`eme 30]{Dumasthese} to give the following proposition.

\begin{proposition}\label{prop:dumas} Let $F(z)\in\mathbb{C}[[z]]$. Then $F(z)$ is $k$-regular if and only if $F(z)$ satisfies some functional equation \eqref{MFE} such that all of the zeros of $a_0(z)$ are either zero or roots of unity of order not coprime to $k$.
\end{proposition}

\noindent Note that the functional equation alluded to in the above proposition need not be minimal.

To prove Theorem \ref{main}, we will show that if $F(z)$ is $k$-regular satisfying \eqref{MFE}, then, essentially, one can `remove' all of the zeros of $a_0(z)$ that are roots of unity. We then show that, after dividing by an appropriate power of $z$, the resulting function satisfies another Mahler-type functional equation with $a_0(z)=1$, but is not necessarily $k$-Becker, since {\em it may not be a power series}.

This line of reasoning is inspired by a recent paper of Kisielewski \cite{K2017}, who considered Becker's conjecture for a subclass of regular functions. Indeed, Kisielewski \cite[Proposition~2]{K2017} showed that Becker's conjecture holds for every $k$-regular function $F(z)$ satisfying a functional equation \eqref{MFE} of minimal degree $d$ such that $a_0(z)$ has no zeros that are roots of unity; specifically, for a function $F(z)$ in this class, he showed there exists a $k$-regular rational function $R(z)$ such that $F(z)/R(z)$ is $k$-Becker; his result is purely existential concerning the rational function $R(z)$. In comparison, Theorem \ref{main} has the following corollary in this context.

\begin{corollary} Suppose that $F(z)$ is a $k$-regular function satisfying a functional equation \eqref{MFE} of minimal degree $d$ such that $a_0(0)\neq 0$ and $a_0(z)$ has no zeros that are roots of unity. Then $F(z)$ is $k$-Becker.
\end{corollary}

This paper is outlined as follows. Section \ref{prelims} contains preliminary results that will be needed in Section \ref{proof}, which contains the proof of Theorem \ref{main}. Section \ref{further} contains justification that Theorem \ref{main} is the best-possible resolution of Becker's conjecture; in particular, in that section, we give an example of a $k$-regular function $F(z)$ such that for any rational function $R(z)$, the function $R(z)F(z)$ cannot simultaneously be a power series and satisfy the conclusion of Becker's conjecture. Finally, in Section~\ref{structure} we prove Proposition \ref{prop:dumas}.

\section{Preliminaries}\label{prelims}

We require the following definition and lemmas.
A general form of Lemma \ref{ch6:Cartier}(a) is given by Dumas \cite[Lemma 4, p.~20]{Dumasthese}
and Lemma~\ref{ch6:Cartier} was known at least to Allouche \cite[p.~255]{A1987}.
For an English reference, we refer the reader to the work of Becker \cite[Lemma~2]{B1994}.
Lemma~\ref{cartvec} is a restatement of Allouche and Shallit \cite[Theorem 2.2]{AS1992}, also provided by Becker \cite[Lemma~3]{B1994}.
For its part, Lemma~\ref{Kislemma8} is due to Kisielewski \cite[Lemma~8]{K2017}.

\begin{definition}\label{cartier} Let $C(z)=\sum_{n\geqslant 0}c(n)z^n$. Given a positive integer $k\geqslant 2$, for each $i\in\{0,\ldots,k-1\}$, we define the {\em Cartier operator} $\Lambda_i:\mathbb{C}[[z]]\to \mathbb{C}[[z]]$ by $$\Lambda_i(C)(z)=\sum_{n\geqslant 0}c(kn+i)z^n.$$
\end{definition}

\begin{lemma}[Dumas \cite{Dumasthese}, Allouche \cite{A1987}]\label{ch6:Cartier} Let $F(z),G(z)\in\mathbb{C}[[z]]$. For $i=0,\ldots,k-1$ we have
\begin{enumerate}
\item[(a)] $\Lambda_i(F(z^k)G(z))=F(z)\Lambda_i(G(z))$, and
\item[(b)] $F(z)=\sum_{i=0}^{k-1} z^i\Lambda_i(F)(z^k).$
\end{enumerate}
\end{lemma}

In Lemma \ref{ch6:Cartier}, $\Lambda_i(F)(z^k)$ is understood as $\Lambda_i(F(z))$ evaluated at $z^k$, so that if we write $F(z)=\sum_{n\geqslant 0} f(n)z^n$, then $\Lambda_i(F)(z^k)=\sum_{n\geqslant 0} f(kn+i)z^{kn}.$

\begin{lemma}[Allouche and Shallit \cite{AS1992}]\label{cartvec}  The function $F(z)\in\mathbb{C}[[z]]$ is $k$-regular if and only if the $\mathbb{C}$-vector space $$V:=\left\langle\left\{\Lambda_{r_n}\cdots \Lambda_{r_1}(F)(z): 0\leqslant r_i<k,\ n\in\mathbb{N}\right\}\right\rangle_\mathbb{C}$$ is finite-dimensional.
\end{lemma}

If one lets $W$ denote the finitely generated $\mathbb{C}[z]$-submodule of the field of Laurent power series $\mathbb{C}((z))$ spanned by the finite-dimensional $\mathbb{C}$-vector space $V$, then $W$ has the property that $$W\subseteq \sum_{h(z)\in W} \mathbb{C}[z] h(z^k).$$  To see this, we let $\{F(z)=h_1(z),\ldots ,h_r(z)\}$ be a basis for $V$.  Then notice that for $i=0,\ldots ,k-1$, we have
$$\Lambda_i(h_j)(z) = \sum_{\ell=1}^{r} c_{i,j,\ell} h_{\ell}(z)$$ for some constants $c_{i,j,\ell}\in \mathbb{C}$.  An application of Lemma \ref{ch6:Cartier}(b) then gives that
\begin{equation}\label{hch}h_j(z) = \sum_{\ell=1}^r \left( \sum_{i=0}^{k-1} c_{i,j,\ell} z^i\right) h_{\ell}(z^k),\end{equation} which gives the desired claim.

In fact, in the case that the dimension of the vector space $V$ is $r<\infty$, we have that $F(z)$ is a $k$-Mahler function of degree at most $r$. To see this, we observe that \eqref{hch} can be written as $${\bf h}(z)={\bf A}(z){\bf h}(z^k),$$ where ${\bf h}(z):=[F(z)=h_1(z), h_2(z),\ldots,h_r(z)]^T$ and ${\bf A}(z)\in\mathbb{C}[z]^{r\times r}.$ Now we let ${\bf A}^{(i)}(z)={\bf A}(z){\bf A}(z^k)\cdots {\bf A}(z^{k^{i-1}})$, where we use the convention that ${\bf A}^{(0)}(z)$ is the identity. So for $i\in\{0,1,\ldots,r\}$, we have ${\bf h}(z^{k^i})={\bf A}^{(r-i)}(z^{k^i}){\bf h}(z^{k^r}).$ Left multiplying by the vector $e_1^T$ and using the fact that $h_1(z)=F(z)$ we obtain equations \begin{equation}\label{Fuh}F(z^{k^i})={\bf u}_i(z){\bf h}(z^{k^r}),\end{equation} for some ${\bf u}_i(z)\in\mathbb{C}[z]^{1\times r}$. Since we have $r+1$ vectors ${\bf u}_0(z),\ldots,{\bf u}_{r}(z)$, we have a nontrivial linear dependence; that is, there are polynomials $p_0(z),\ldots,p_r(z)$, not all zero, such that $$p_0(z){\bf u}_0(z)+\cdots+p_r(z){\bf u}_{r}(z)=0.$$ Combining this with \eqref{Fuh} shows that $F(z)$ satisfies the functional equation $$p_i(z)F(z^{k^i})+p_{i+1}(z)F(z^{k^{i+1}})+\cdots+p_r(z)F(z^{k^r})=0,$$ where $i$ is the smallest index such that $p_i(z)\neq 0$. But this implies that $F(z)$ is a $k$-Mahler function of degree at most $r-i\leqslant r$; see, e.g., Becker's argument \cite[p.~273]{B1994} of taking successive sections to reduce the smallest index.

\begin{lemma}[Kisielewski \cite{K2017}]\label{Kislemma8} Let $c(z)\in\mathbb{C}(z)$, $\alpha\in\mathbb{C}\setminus\{0\}$, and $\nu_\alpha(c(z))$ be the order of the zero of $c(z)$ at $z=\alpha$. There is an $r\in\{0,\ldots,k-1\}$ such that $\nu_\alpha\left(\Lambda_r (c)(z^k)\right)\leqslant \nu_\alpha\left(c(z)\right).$
\end{lemma}

We will use the functional equation \eqref{MFE} in a slightly different form. For a Mahler function satisfying \eqref{MFE} of degree $d$, setting $${\bf F}(z):=[F(z), F(z^k),\ldots,F(z^{k^{d-1}})]^T$$ and $${\bf A}(z):=\left[\begin{matrix} -\frac{a_1(z)}{a_0(z)}\ \   -\frac{a_2(z)}{a_0(z)}\ \  \ \cdots & -\frac{a_d(z)}{a_0(z)} \\  {\bf I}_{(d-1)\times (d-1)} & {\bf 0}_{(d-1)\times 1}\end{matrix}\right],$$ we have \begin{equation}\label{FAF}{\bf F}(z)={\bf A}(z){\bf F}(z^k).\end{equation}
We will be specifically interested in the matrices \begin{equation}\label{Bz} {\bf B}_n(z):={\bf A}(z){\bf A}(z^k)\cdots {\bf A}(z^{k^{n-1}}).\end{equation} Note that ${\bf F}(z)={\bf B}_n(z){\bf F}(z^{k^n})$ for every $n\geqslant 1$. In what follows, for $i=1,\ldots,d$, we write $$e_i:=\left[{\bf 0}_{1\times (i-1)}\ 1\ {\bf 0}_{1\times (d-i)}\right]^T.$$

Kisielewski's lemma above states that a Cartier operator can be used to (possibly) reduce the order of a zero. We use this result in the following lemma to find an upper bound on the order of certain poles of the matrices ${\bf B}_n(z).$

\begin{lemma}\label{Buniform} Suppose $F(z)$ is $k$-regular, ${\bf B}_n(z)$ is as defined in \eqref{Bz}, and $\xi$ is a root of unity such that $\xi^k=\xi$. Then the poles at $z=\xi$ of the entries of the matrices $\{{\bf B}_n(z):n\geqslant 1\}$ have uniformly bounded order. In particular, there is a polynomial $h(z)\in\mathbb{C}[z]$ such that for each $n$ the matrix $h(z)\cdot {\bf B}_n(z)$ has polynomial entries.
\end{lemma}

\begin{proof} For each $i\in\{1,\ldots,d\}$ and each $n\in\mathbb{N}$, set $$c_{i,n}(z):=e_1^T {\bf B}_n(z) e_i.$$ Then for each $n$ we have \begin{equation}\label{FcFn} F(z)=\sum_{i=1}^d c_{i,n}(z)F(z^{k^{i+n-1}}).\end{equation} If we apply $n$ Cartier operators to \eqref{FcFn}, we have \begin{equation}\label{CartFcFn} \Lambda_{r_n}\cdots \Lambda_{r_1}(F)(z)=\sum_{i=1}^d \Lambda_{r_n}\cdots \Lambda_{r_1}(c_{i,n})(z)\cdot F(z^{k^{i-1}}).\end{equation} Since $d$ here is minimal, the functions $F(z),\ldots,F(z^{k^{d-1}})$ are linearly independent over $\mathbb{C}(z)$.

Now suppose $F(z)$ is $k$-regular. Since the $\mathbb{C}$-vector space $V$ defined in Lemma \ref{cartvec} is finite-dimensional, its finite number of generators are of the form $\sum_{i=1}^d h_i(z) F(z^{k^{i-1}}),$ for some rational functions $h_i(z)$. Since, as we run over a finite generating set, the $h_i(z)$ that occur are a finite number of rational functions and the functions $F(z),\ldots,F(z^{k^{d-1}})$ are linearly independent over $\mathbb{C}(z)$, there is a nonzero polynomial $h(z)$ such that $$V\subseteq h(z)^{-1}\sum_{i=1}^d \mathbb{C}[z]F(z^{k^{i-1}}).$$ This,
requires for every $i\in\{1,\ldots,d\}$, $n\in\mathbb{N}$ and choice of Cartier operators that the inequality \begin{equation}\label{cartup}\nu_\xi\left(h(z)^{-1}\right)\leqslant\nu_\xi\left(\Lambda_{r_n}\cdots \Lambda_{r_1}(c_{i,n})(z)\right)\end{equation} holds between orders of the zeros at $z=\zeta$. Since $\xi^k=\xi$, for each rational function $c(z)$, we have $\nu_\xi\left(c(z)\right)=\nu_\xi\left(c(z^k)\right)$. By Lemma \ref{Kislemma8} and \eqref{cartup}, for each $n\in\mathbb{N}$ there is a choice of Cartier operators $\Lambda_{r_1},\ldots, \Lambda_{r_n}$ such that, for zero orders, $$ \nu_\xi\left(h(z)\right)\leqslant \nu_\xi\left(\Lambda_{r_n}\cdots \Lambda_{r_1}(c_{i,n})(z)\right) =\nu_\xi\big(\Lambda_{r_n}\cdots \Lambda_{r_1}(c_{i,n})(z^{k^n})\big)\leqslant \nu_\xi\left(c_{i,n}(z)\right).
$$ Thus the poles of the entries $c_{i,n}(z)$ of the first row of the matrices ${\bf B}_n(z)$ at $z=\xi$ have uniformly bounded order; specifically, they are bounded above by $\nu_\zeta(h(z))$.

It remains now to show this for the rest of the rows, but this follows due to the structure of the matrix ${\bf A}(z)$. In fact, consider the $(i,j)$ entry of the matrix ${\bf B}_n(z)$ for some $i\in\{2,\ldots,d\}$. Using the definition of ${\bf B}_n(z)$, we have $$e_i^T{\bf B}_n(z)e_j=e_i^T{\bf A}(z){\bf B}_{n-1}(z^k)e_j.$$ Now, $e_i^T{\bf A}(z)=e_{i-1}^T.$ So, \begin{equation}\label{eidown}\nu_\xi(e_i^T{\bf B}_n(z)e_j)=\nu_\xi(e_{i-1}^T{\bf B}_{n-1}(z^k)e_j)=\nu_\xi(e_{i-1}^T{\bf B}_{n-1}(z)e_j),\end{equation} where the last equality uses, again, the facts that $\xi^k=\xi$ and for every rational function $\nu_\xi\left(c(z)\right)=\nu_\xi\left(c(z^k)\right)$. Applying \eqref{eidown} $i-1$ times, we have \begin{equation}\label{e1down}\nu_\xi(e_i^T{\bf B}_n(z)e_j)=\nu_\xi(e_{1}^T{\bf B}_{n-i+1}(z)e_j),\end{equation} which immediately implies the desired result.
\end{proof}

\begin{proposition}\label{prop0} If $F(z)$ is $k$-Mahler satisfying \eqref{MFE} of degree $d$, $a_0(\xi)=0$ for some root of unity $\xi$ with $\xi^k=\xi$, and $\gcd(a_0(z),a_1(z),\ldots,a_d(z))=1$, then $F(z)$ is not $k$-regular.
\end{proposition}

\begin{proof} Towards a contradiction, assume that $F(z)$ is $k$-regular and suppose that $\xi$ is a root of unity with $\xi^k=\xi$ such that $a_0(\xi)=0$. Using Lemma \ref{Buniform}, let $Y$ denote the minimal uniform bound of the order of the poles at $z=\xi$ of $\{{\bf B}_n(z):n\geqslant 1\}$, and note that $Y>0$ since $\gcd(a_0(z),a_1(z),\ldots,a_d(z))=1$.

We examine the first row of ${\bf B}_1(z)={\bf A}(z)$. In particular, set $$N:=\min\Big\{i\in\{1,\ldots,d\}: \nu_\xi\left(\frac{a_i(z)}{a_0(z)}\right)\leqslant \nu_\xi\left(\frac{a_j(z)}{a_0(z)}\right) \mbox{for all $j\in\{1,\ldots,d\}$}\Big\},$$ and note, again using $\gcd(a_0(z),a_1(z),\ldots,a_d(z))=1$, that \begin{equation}\label{X}X:=-\nu_\xi\left(\frac{a_N(z)}{a_0(z)}\right)>0.\end{equation} By the minimality of $N$, we have both $$\nu_\xi(a_i(z)/a_0(z))>-X\quad\mbox{for}\quad i<N,$$ and $$\nu_\xi(a_i(z)/a_0(z))\geqslant -X\quad\mbox{for}\quad i>N.$$

Since ${\bf B}_1(z)={\bf A}(z)$ has only constant entries outside of its first row, \eqref{eidown} and \eqref{e1down} imply there is some minimal $n$, say $m$, for which the maximal order of the pole at $z=\xi$ of the entries of ${\bf B}_m(z)$ is $Y$, occurs in the first row of ${\bf B}_m(z)$, say in the $(1,J)$ entry, and all of the other rows have entries with poles at $z=\xi$ of order strictly less than $Y$. That is, specifically, within the $J$-th column of ${\bf B}_m(z)$, we have \begin{equation}\label{Y}\nu_\xi\left(e_1^T{\bf B}_m(z)e_J\right)=-Y<0\quad\mbox{and}\quad \nu_\xi\left(e_i^T{\bf B}_m(z)e_J\right)>-Y,\end{equation} for each $i\in\{2,\ldots,d\}$.

Now, define the rational functions $b_1(z),\ldots,b_d(z)$ by $${\bf B}_{m+N-1}(z^k)e_J=\left[\begin{matrix} b_1(z)\ \cdots\ b_d(z)\end{matrix}\right]^T,$$ and note that by \eqref{e1down} we have, since $\xi^k=\xi$ and for every rational function $\nu_\xi\left(c(z)\right)=\nu_\xi\left(c(z^k)\right)$, that \begin{equation}\label{Y}-Y=\nu_\xi(e_1^T{\bf B}_m(z)e_J)=\nu_\xi(e_{N}^T{\bf B}_{m+N-1}(z)e_J)=\nu_\xi(b_N(z)).\end{equation} By the minimality of $m$, we have $$\nu_\xi(b_i(z))>-Y\quad\mbox{for}\quad i>N,$$ and trivially $$\nu_\xi(b_i(z))\geqslant -Y\quad\mbox{and}\quad i<N.$$

Let us now see how we can put together the results of the previous paragraphs in order to obtain the desired result. Consider the first entry of the $J$th column of ${\bf B}_{m+N}(z)$. We have \begin{align} \nonumber e_1^T{\bf B}_{m+N}(z)e_J
&=e_1^T{\bf A}(z){\bf B}_{m+N-1}(z^k)e_J\\
\nonumber &=\left[\begin{matrix} -\frac{a_1(z)}{a_0(z)}\ \   -\frac{a_2(z)}{a_0(z)}\ \   \cdots & -\frac{a_d(z)}{a_0(z)}\end{matrix}\right]{\bf B}_{m+N-1}(z^k)e_J\\
\label{finalnumber}&=-\sum_{i=1}^{N-1}\frac{a_i(z)b_i(z)}{a_0(z)} -\frac{a_N(z)b_N(z)}{a_0(z)}-\sum_{i=N+1}^{d}\frac{a_i(z)b_i(z)}{a_0(z)}.\end{align}
For $i\neq N$, using the comments immediately below Equations \eqref{Y} and \eqref{X}, respectively, we have \begin{equation}\label{notN}\nu_\xi\left(\frac{a_i(z)b_i(z)}{a_0(z)}\right)=\nu_\xi\left(\frac{a_i(z)}{a_0(z)}\right)+\nu_\xi\left(b_i(z)\right)>-X-Y,\end{equation} since both $\nu_\xi\left({a_i(z)/a_0(z)}\right)>-X$ for $i\in\{1,\ldots,N-1\}$ and $\nu_\xi\left(b_i(z)\right)>-Y$ for $i\in\{N+1,\ldots,d\}$. Also, by \eqref{Y} and \eqref{X}, we have \begin{equation}\label{N}\nu_\xi\left(\frac{a_N(z)b_N(z)}{a_0(z)}\right)=\nu_\xi\left(\frac{a_N(z)}{a_0(z)}\right)+\nu_\xi\left(b_N(z)\right)=-X-Y.\end{equation} Hence, using \eqref{notN} and \eqref{N}, Equation \eqref{finalnumber} gives the inequality  $$\nu_\xi\left(e_1^T{\bf B}_{m+N}(z)e_J\right)=-X-Y<-Y,$$ contradicting that $Y$ is a uniform bound on the pole order at $z=\xi$ over all $e_i^T{\bf B}_n(z)e_j$. Thus $F(z)$ is not $k$-regular.
\end{proof}

\begin{corollary}\label{zetaN=zeta} Suppose $F(z)$ is $k$-regular satisfying \eqref{MFE} of degree $d$. Let $\mathfrak{I}$ be the ideal of polynomials $p(z)$ such that $$p(z)F(z) \in \sum_{j\geqslant 1} \mathbb{C}[z]F(z^{k^j})$$ and let $q(z)$ be a generator for $\mathfrak{I}$. If $\xi$ is a zero of $q(z)$ such that $\xi^{k^M}=\xi$ for some $M\geqslant 1$, then $\xi=0$.
\end{corollary}

\begin{proof} Suppose that there exists $M\geqslant 1$ and $\xi$ such that $q(\xi)=0$ with $\xi^{k^M}=\xi$ and $\xi\neq 0$.  Let $$q_0(z)F(z) + q_1(z) F(z^{k^M})+ \cdots + q_D(z) F(z^{k^{MD}}) = 0$$ be a relation with $q_0(z)\neq 0$, $\gcd(q_0(z),q_1(z),\ldots,q_D(z))=1$ and $D$ minimal.  Then $q(z)$ divides $q_0(z)$ and so $q_0(\xi)=0$.  But $\xi^{k^M}=\xi$ and by Proposition \ref{prop0} this contradicts the fact that $F(z)$ is $k^M$-regular. Since $F(z)$ is $k^M$-regular if and only if it is $k$-regular \cite[Theorem~2.9]{AS1992}, this proves the corollary.
\end{proof}

\begin{lemma}\label{lem:norou} Let $F(z)$ be a $k$-regular power series satisfying \eqref{MFE} of degree $d$. Then there exist a polynomial $Q(z)$ with $Q(0)\neq 0$ such that $1/Q(z)$ is $k$-regular and a nonnegative integer $\gamma$ such that  $G(z):={F(z)}/{z^{\gamma } Q(z)}$ satisfies a Mahler-type functional equation $$q_0(z)G(z)+q_1(z)G(z^k)+\cdots+q_d(z)G(z^{k^d})=0,$$ of degree $d$, $q_i(z)\in\mathbb{C}[z]$ with $q_0(0)\neq 0$, and if $\zeta$ is a zero of $q_0(z)$ that is a root of unity then there is some $M\geqslant 1$ such that $\zeta^{k^M}=\zeta$.
\end{lemma}

The proof of Lemma \ref{lem:norou} requires the following characterisation of Mahler functions due to Dumas \cite[Theorem 31, p.~153]{Dumasthese}; see also Coons and Spiegelhofer \cite{CS2018} for a proof of Dumas's result in English.

\begin{theorem}[Structure Theorem of Dumas]\label{lem:DAB} A $k$-Mahler function is the quotient of a series and an infinite product which are $k$-regular. That is, if $F(z)$ is the solution of the Mahler functional equation $$a_0(z)F(z)+a_1(z)F(z^k)+\cdots+a_d(z)F(z^{k^d})=0,$$ where $a_0(z)a_d(z)\neq 0$, the $a_i(z)$ are polynomials, then there exists a $k$-regular series $J(z)$ such that $$F(z)=\frac{J(z)}{\prod_{j\geqslant 0}\Gamma(z^{k^j})},$$ where $a_0(z)=\rho z^{\delta}\Gamma(z)$, with $\rho\neq 0$ and $\Gamma(0)=1$.
\end{theorem}

\begin{proof}[Proof of Lemma \ref{lem:norou}] Suppose that $F(z)$ is a $k$-Mahler power series of degree $d$ satisfying \eqref{MFE}. Let $A$ be the set of roots of unity $\zeta$ such that $a_0(\zeta)=0$ and there does not exist $M\geqslant 1$ such that $\zeta^{k^M}=\zeta$ and set $\nu_\zeta(a_0):=\nu_\zeta(a_0(z))$. For each $\zeta\in A$, the sequence $\{\zeta^{k^i}\}_{i\geqslant 0}$ is eventually periodic, so that there is an $M_\zeta$ such that $\zeta^{k^{2M_\zeta}}=\zeta^{k^{M_\zeta}}$. Note that this then implies that $\zeta^{k^{pM_\zeta}}=\zeta^{k^{M_\zeta}}$ for all $p\geqslant 1$. Now, set $$N:=\prod_{\zeta\in A}M_\zeta,$$ so that $\zeta^{k^{2N}}=\zeta^{k^N}$ for all $\zeta\in A$. Define the polynomial $Q(z)$ by $$Q(z):=\prod_{\zeta\in A}\prod_{j=0}^{N-1}(1-z^{k^j}\overline{\zeta}^{k^N})^{\nu_\zeta(a_0)}.$$

Then $$\frac{Q(z^k)}{Q(z)}=\prod_{\zeta\in A}\left(\frac{1-z^{k^N}\overline{\zeta}^{k^N} }{1-z\overline{\zeta}^{k^N} }\right)^{\nu_\zeta(a_0)}\in\mathbb{C}[z],$$ since for each $\zeta\in A$, $$1- z^{k^N}\overline{\zeta}^{k^N}=1-(z\overline{\zeta}^{k^N})^{k^N}=(1-z\overline{\zeta}^{k^N} )(1+(z\overline{\zeta}^{k^N})+\cdots+(z\overline{\zeta}^{k^N} )^{k^N-1}).$$ But also for each $\xi\in A$, \begin{align*}\frac{Q(z^k)}{Q(z)}(1-& z\overline{\xi}^{k^N} )^{\nu_\xi(a_0)}\\
&=\left(1-(z\overline{\xi} )^{k^N}\right)^{\nu_\xi(a_0)}\prod_{\zeta\in A\setminus\{\xi\}}\left(\frac{1-z^{k^N}\overline{\zeta}^{k^N} }{1-z\overline{\zeta}^{k^N} }\right)^{\nu_\zeta(a_0)}\\
&=(1-z\overline{\xi} )^{\nu_\xi(a_0)}\left(\sum_{j=0}^{k^N-1}(z\overline{\xi} )^j\right)^{\nu_\xi(a_0)}\prod_{\zeta\in A\setminus\{\xi\}}\left(\frac{1-z^{k^N}\overline{\zeta}^{k^N} }{1-z\overline{\zeta}^{k^N} }\right)^{\nu_\zeta(a_0)}
\end{align*} is a polynomial. Since $\overline{\xi}\neq\overline{\xi}^{k^N}$, we have that $(1-z\overline{\xi})^{\nu_\xi(a_0)}$ divides the polynomial $Q(z^k)/Q(z)$. As this is true for all $\xi\in A$, there is a polynomial $h(z)$ such that \begin{equation}\label{QoverQpoly}\frac{Q(z^k)}{Q(z)}=\left(\prod_{\zeta\in A}(1-z\overline{\zeta})^{\nu_\zeta(a_0)}\right)\cdot h(z).\end{equation} Set $P(z):=\prod_{\zeta\in A}(1-z\overline{\zeta})^{\nu_\zeta(a_0)}.$ Then \eqref{QoverQpoly} shows that the polynomial $Q(z)$ satisfies \begin{equation}\label{Qzh}Q(z)=\left(\prod_{j\geqslant 0} P(z^{k^j})\right)^{-1}\left(\prod_{j\geqslant 0} h(z^{k^j})\right)^{-1}.\end{equation}
We factor $a_0(z)=cz^\gamma\Gamma(z)=cz^\gamma a(z)P(z)$, where $a(\zeta)\neq 0$ for every $\zeta\in A$ and $a(0)=1$.  By Proposition \ref{prop0}, since $F(z)$ is $k$-regular, $a_0(1)\neq 0$. Then using Theorem~\ref{lem:DAB} and \eqref{Qzh}, there is a $k$-regular function $J(z)$ such that  $$F(z)=\frac{J(z)}{\prod_{j\geqslant 0}\Gamma(z^{k^j})}=\frac{J(z)\cdot Q(z)\cdot\prod_{j\geqslant 0} h(z^{k^j})}{\prod_{j\geqslant 0}a(z^{k^j})}.$$

Setting $H(z):=J(z)\prod_{j\geqslant 0} h(z^{k^j})$, we have that $$G(z):=\frac{F(z)}{z^{\gamma} Q(z)}=\frac{H(z)}{z^{\gamma} \prod_{j\geqslant 0}a(z^{k^j})},$$ where $1/Q(z)$ is $k$-regular by an above-mentioned result of Becker \cite[Theorem 2]{B1994} and \eqref{Qzh}. To build the functional equation for $G(z)$, we start with the functional equation \eqref{MFE} for $F(z)$ of degree $d$, and divide by $z^{2\gamma}Q(z)P(z)$ to get \begin{equation}\label{FPQ} c\cdot a(z) \frac{F(z)}{z^\gamma Q(z)}+\sum_{i=1}^d \frac{a_i(z)F(z^{k^i})}{z^{2\gamma}Q(z)P(z)}=0.\end{equation} These coefficients, for $i=1,\ldots,d$, satisfy $$\frac{a_i(z)}{z^{2\gamma}Q(z)P(z)}=\frac{z^{\gamma(k^i-2)}a_i(z)Q(z^{k^i})}{z^{k^i\gamma}Q(z^{k^i})Q(z)P(z)}=\frac{z^{\gamma(k^i-2)}a_i(z)}{z^{k^i\gamma}Q(z^{k^i})}\cdot \frac{Q(z^{k})}{Q(z)P(z)}\prod_{j=2}^i\frac{Q(z^{k^j})}{Q(z^{k^{j-1}})},$$ where as usual, the empty product is taken to be equal to $1$. By \eqref{QoverQpoly}, we have $Q(z^{k^j})/Q(z^{k^{j-1}}) =P(z^{k^{j-1}})h(z^{k^{j-1}}),$ so continuing the above equality gives \begin{equation}\label{qia}\frac{a_i(z)}{z^{2\gamma}Q(z)P(z)}= \frac{z^{\gamma(k^i-2)}a_i(z)}{z^{k^i\gamma}Q(z^{k^i})}\cdot h(z)\prod_{j=2}^i\left(P(z^{k^{j-1}})h(z^{k^{j-1}})\right)=\frac{q_i(z)}{z^{k^i\gamma}Q(z^{k^i})},\end{equation} where $q_i(z)$ is the polynomial $$q_i(z):=z^{\gamma(k^i-2)}a_i(z)h(z)\prod_{j=2}^i\left(P(z^{k^{j-1}})h(z^{k^{j-1}})\right).$$ Finally, defining $q_0(z):=c\cdot a(z)$, substituting the result of \eqref{qia} into \eqref{FPQ} and using the definition of $G(z)$, we have that $G(z)$ satisfies the functional equation $$q_0(z)G(z)+q_1(z)G(z^k)+\cdots+q_d(z)G(z^{k^d})=0.$$ Here $G(z)$ inherits the degree $d$ from $F(z)$, $q_0(0)c\cdot a(0)=c\neq 0$ and $q_0(z)$ inherits the desired root properties from $a(z)$. This finishes the proof of the lemma.
\end{proof}

Our method of proof of Lemma \ref{lem:norou} is inspired by remarks of Becker \cite[p.~279]{B1994} as well as an argument of Adamczewski and Bell \cite[Proposition 7.2]{ABkl}.

\section{Proof of the main result}\label{proof}

\begin{proof}[Proof of Theorem \ref{main}] Suppose that $F(z)$ is a $k$-regular function satisfying \eqref{MFE} of degree $d$. By Lemma \ref{lem:norou}, there exist a polynomial $Q(z)$ with $Q(0)=1$ such that $1/Q(z)$ is $k$-regular and a nonnegative integer $\gamma$ such that the $k$-Mahler function $G(z):=z^{-\gamma} F(z)/Q(z)$ satisfies a Mahler functional equation \begin{equation}\label{q0G}q_0(z)G(z)+q_1(z)G(z^k)+\cdots+q_d(z)G(z^{k^d})=0\end{equation} of minimal degree $d$, $q_i(z)\in\mathbb{C}[z]$, and $q_0(z)$ has the property that $q_0(0)\neq 0$ and if $q_0(\zeta)=0$ with $\zeta$ a root of unity then there is some $M\geqslant 1$ such that $\zeta^{k^M}=\zeta$.

We let $\mathfrak{I}$ be the ideal of polynomials $p(z)$ such that $$p(z)G(z) \in \sum_{j\geqslant 1} \mathbb{C}[z]G(z^{k^j}).$$ Since $G$ is $k$-Mahler, $\mathfrak{I}$ is nonzero, and we let $q(z)$ be a generator for $\mathfrak{I}$ whose leading coefficient is $1$.  Since $q_0(z)\in \mathfrak{I}$, we have $q(z)$ divides $q_0(z)$ and so we have $q(0)\neq 0$. By Corollary \ref{zetaN=zeta}, if $\zeta$ is a zero of $q(z)$, then there does not exist $M\geqslant 1$ such that $\zeta^{k^M}=\zeta$. Thus if $\zeta$ is a root of $q(z)$ then $\zeta$ cannot be a root of unity, since we have shown that any zero of $q_0(z)$ that is a root of unity must satisfy $\zeta^{k^M}=\zeta$ for some $M\geqslant 1$, and we have also shown that each zero $\zeta$ of $q_0(z)$ that is a root of unity has the property that there is no $M\geqslant 1$ such that $\zeta^{k^M}=\zeta$. Hence $q(z)$ has no zeros that are either zero or a root of unity. Thus $G(z)$ has the property that there is a relation $$q(z)G(z) \in \sum_{j\geqslant 1} \mathbb{C}[z]G(z^{k^j})$$ with $q(z)$ having no zeros that are roots of unity and with $q(0)\neq 0$.

We now claim that $q(z)=1$. To see this, suppose that $q(z)$ is non-constant.  Then there is a nonzero complex number $\lambda$ that is not a root of unity such that $q(\lambda)=0$.  Since $G(z)$ is $k$-regular, the $\mathbb{C}$-vector space spanned by all elements of the form $\Lambda_{r_m}\cdots \Lambda_{r_0}(G)(z)$ (including also $G(z)$) is finite-dimensional.  Moreover, its basis elements are of the form $\sum_{i=1}^d h_i(z) G(z^{k^{i-1}}),$ for some rational functions $h_i(z)$, where $d$ is the degree of the Mahler function $G$. Since, as we run over a basis, only finitely many rational functions $h_i(z)$ occur and since the functions $F(z),\ldots,F(z^{k^{d-1}})$ are linearly independent over $\mathbb{C}(z)$, there is a nonzero polynomial $h(z)$ such that $$V\subseteq h(z)^{-1}\sum_{i=1}^d \mathbb{C}[z]G(z^{k^{i-1}}),$$ where $V$ is the $\mathbb{C}$-vector space defined in Lemma \ref{cartvec}. Now since $\lambda$ is nonzero and is not a root of unity, there exists some positive integer $N$ such that $\lambda^{k^N}$ is not a zero of $h(z)$.
Repeatedly using the Mahler Equation \eqref{q0G}, we obtain a relation of the form $$Q(z) G(z) = \sum_{j=1}^{d} Q_j(z) G(z^{k^{N+j-1}})$$ with $Q(z), Q_1(z),\ldots ,Q_d(z)$ polynomials and $Q(z)\neq0$ and $\gcd(Q(z),Q_1(z),\ldots,$ $Q_d(z))=1$.  Since $Q(z)\in \mathfrak{I}$, we see that $q(z)$ divides $Q(z)$ and so $\lambda$ is a root of $Q(z)$.

Now we write $$G(z)  = \sum_{j=1}^d R_j(z) G(z^{k^{N+j-1}}),$$ with $$R_j(z):=Q_j(z)/Q(z).$$ Moreover, since $\gcd(Q(z),Q_1(z),\ldots, Q_d(z))=1$, we have $\nu_{\lambda}(R_{\ell}(z))<0$ for some $\ell\in \{1,\ldots ,d\}$.
By Lemma \ref{Kislemma8}, there exists $(r_1,\ldots ,r_N)\in \{0,1,\ldots ,k-1\}^N$ such that
$$\nu_{\lambda}(\Lambda_{r_N} \cdots  \Lambda_{r_1}(R_{\ell})(z^{k^N}))\leqslant \nu_{\lambda}(R_{\ell}(z)) < 0.$$
Thus
$$\nu_{\lambda^{k^n}} (\Lambda_{r_N} \cdots \Lambda_{r_1}(R_{\ell})(z)) < 0.$$  Now set
$$T_j(z) :=\Lambda_{r_N} \cdots  \Lambda_{r_1}(R_j)(z)\quad \mbox{for}\quad j=1,\ldots ,d.$$
Then
$\Lambda_{r_N} \cdots  \Lambda_{r_1}(G)(z)\in V$ and so
$$\sum_{j=1}^d T_j(z) G(z^{k^{j-1}}) \in V.$$  Since $G(z),\ldots ,G(z^{k^{d-1}})$ are linearly independent over $\mathbb{C}(z)$ we must have that $h(z)T_j(z) \in \mathbb{C}[z]$ for $j=1,\ldots, d$.  But $\nu_{\lambda^{k^N}}(h(z))=0$ and so $\nu_{\lambda^{k^N}}(h(z)T_{\ell}(z))<0$, which contradicts the fact that $h(z)T_j(z)$ must be a polynomial, giving the claim.  It follows that $q(z)=1$.

Specifically, $1\in \mathfrak{I}$ and so $$G(z)\in \sum_{j\ge 1} \mathbb{C}[z]G(z^{k^j}),$$ which says that $G(z)$ satisfies a Mahler functional equation of the form \eqref{MFE} with $a_0(z)=1$.
This finishes the proof of Theorem \ref{main}.
\end{proof}

\section{Optimality of the Theorem \ref{main}}\label{further}

The careful reader will notice that, while we prove Becker's conjecture completely, the resulting function $F(z)/z^\gamma Q(z)$ that satisfies a Mahler-type functional equation \eqref{MFE} with $a_0(z)=1$ is not necessarily a power series, so that (strictly speaking) it is neither $k$-regular nor $k$-Becker. One may argue, that probably the field of Laurent series is a preferable setting for solutions to \eqref{MFE}, and indeed a result of Dumas's \cite[Th\'eor\`eme 7]{Dumasthese} gives reasonable bounds on the valuation at $z=0$ of the solutions.

\begin{theorem}[Dumas] Let $F(z)$ be a Laurent power series solution to a Mahler-type functional equation \eqref{MFE} of degree $d$.  Then $F(z)\in z^{-\nu}\mathbb{C}[[z]]$, where $$\nu:=\left\lceil\max\left\{\frac{\nu_0(a_d(z))}{k^d},\frac{\nu_0(a_d(z)/a_0(z))}{(k^d-1)}\right\}\right\rceil.$$
\end{theorem}

\noindent In this section, we show, by giving an example, that a stronger variant of Becker's conjecture with the added conclusion that the resulting function $F(z)/R(z)$ is a power series cannot hold; that is, with the currently in-use definitions, such a function is not necessarily $k$-Becker.  We now state this result.
\begin{theorem} Let $k\geqslant 2$ be a natural number.  Then there exists a $k$-regular power series $F(z)$ such that there is no nonzero rational function $R(z)$ with the property that $F(z)R(z)$ is a $k$-Becker power series.
\label{thm: example}
\end{theorem}
We note that this does not contradict the conclusion of Theorem \ref{main}, but merely shows that one must necessarily work in the ring of Laurent power series in order to obtain the conclusion.  More precisely, the examples we give in establishing Theorem~\ref{thm: example} have the property that $F(z)/z$ is $k$-Becker with a pole at $z=0$ and so it has a Laurent power series expansion, but not an expansion in the ring of formal power series around $z=0$; moreover, one must introduce a pole at $z=0$ in order to obtain a $k$-Becker function.

Towards the goal of producing these examples, let $k$ be a natural number that is greater than or equal to two and consider the functional equation
\begin{equation}\label{Afunc} A(z) = (1-z+z^{k-1}) A(z^k) - z^{k^2-k} (1-z)A(z^{k^2}).\end{equation} Then writing $${\bf M}(z):=\left[\begin{matrix} 1-z+z^{k-1} & -z^{k^2-k}(1-z)\\ 1 & 0 \end{matrix}\right]$$ and ${\bf A}(z)=[A(z), A(z^k)]^T$, we have $${\bf A}(z)={\bf M}(z){\bf A}(z^k).$$ Let $H(z)$ be the power series solution of the functional equation \eqref{Afunc} corresponding to the iteration of the matrix ${\bf M}(z)$; that is, set $$H(z):=\lim_{n\to\infty} [1,\,0]{\bf M}(z){\bf M}(z^k)\cdots {\bf M}(z^{k^{n-1}}) [1,\,0]^T.$$ To see that this limit exists, it is enough to notice that $k^2-k\geq1$ so that $${\bf M}(z)=\left[\begin{matrix} 1+O(z) & O(z)\\ 1 & 0 \end{matrix}\right]$$ and then for any $n\geqslant 1$, \begin{align*}{\bf M}(z^{k^{n-1}})\left({\bf M}(z^{k^n})-\left[\begin{matrix} 1 & 0\\ 0 & 1 \end{matrix}\right]\right)&=\left[\begin{matrix} O(1) & O(z^{k^{n-1}})\\ 1 & 0 \end{matrix}\right]\left[\begin{matrix} O(z^{k^n}) & O(z^{k^n})\\ 1 & -1 \end{matrix}\right] \\
	&=\left[\begin{matrix} O(z^{k^{n-1}}) & O(z^{k^{n-1}})\\ O(z^{k^n}) & O(z^{k^n}) \end{matrix}\right].\end{align*}
It then follows that for any $n\geqslant 2$, the difference for consecutive terms within the limit is \begin{align*}[1,\,0]{\bf M}(z)&{\bf M}(z^k)\cdots {\bf M}(z^{k^{n-1}}){\bf M}(z^{k^{n}}) [1,\,0]^T\\
	&\qquad-[1,\,0]{\bf M}(z){\bf M}(z^k)\cdots {\bf M}(z^{k^{n-1}}) [1,\,0]^T\\
	&= [1,\,0]{\bf M}(z){\bf M}(z^k)\cdots {\bf M}(z^{k^{n-2}})\left[\begin{matrix} O(z^{k^{n-1}}) & O(z^{k^{n-1}})\\ O(z^{k^n}) & O(z^{k^n})\end{matrix}\right] [1,\,0]^T\\
	&= O(z^{k^{n-1}}).\end{align*} Here we note that $H(0)=1$. We also note that the function $H_0(z):=1/z$ is a solution to the functional equation \eqref{Afunc}.

We continue by setting
\begin{equation} \label{F0}
F_0(z) := H(z)+\frac{1}{z},\end{equation} which again satisfies \begin{equation}\label{F_1}F_0(z) = (1-z+z^{k-1}) F_0(z^k) - z^{k^2-k} (1-z)F_0(z^{k^2}).\end{equation}  As $H(z)$ is a $k$-Becker power series, it is $k$-regular, thus \begin{equation} \label{eqF} F(z):=zF_0(z)=1+zH(z)\end{equation} is $k$-regular, as the $k$-regular power series form a ring.  We note that $F_0(z)=F(z)/z$ is $k$-Becker and is a Laurent power series.  We show, however, that there does not exist a nonzero rational function $R(z)$ such that $F(z)R(z)$ is a $k$-Becker power series; that is, in order to obtain a $k$-Becker function that is a nonzero rational function multiple of $F(z)$ one must work in the ring of Laurent power series and cannot restrict one's focus to the ring of formal power series.  In order to show the desired result, we first establish two key lemmas. We note that the following lemma can also be proved using the method of Roques \cite{R2018}.

\begin{lemma}\label{independent} Let $k\geqslant 2$ and let $F_0(z)$ be as in Equation \eqref{F0}.  Then the Laurent power series $F_0(z)$ and $F_0(z^k)$ are linearly independent over $\mathbb{C}(z)$.
\end{lemma}

\begin{proof} Suppose not.  Then since $F_0(z)$ is nonzero, there is a rational function $a(z)$ such that $F_0(z^k)/F_0(z)=a(z)$.  We note that $$F_0(z)=\frac{1}{z} + 1 + O(z)$$ and so $$\frac{F_0(z^k)}{F_0(z)} = z^{1-k}( 1-z + O(z^2)).$$  It follows that there are relatively prime polynomials $P(z)$ and $Q(z)$ with $P(0)=Q(0)=1$ such that $a(z) = z^{1-k} P(z)/Q(z)$.  Then since $$F_0(z^{k^2}) = a(z)a(z^k) F_0(z),$$ Equation \eqref{F_1} gives
$$1 = (1-z+z^{k-1}) z^{1-k} \cdot\frac{P(z)}{Q(z)} - z^{k^2-k} (1-z)z^{1-k^2} \cdot\frac{P(z)P(z^k)}{Q(z)Q(z^k)}.$$
Clearing denominators, we see
\begin{equation} \label{PQ}
z^{k-1} Q(z)Q(z^k) = (1-z+z^{k-1}) P(z)Q(z^k) - (1-z)P(z)P(z^k).\end{equation}
In particular, $Q(z^k)$ divides $(1-z)P(z)P(z^k)$ and since $P(z)$ and $Q(z)$ are relatively prime, we then have that
$Q(z^k)$ divides $(1-z)P(z)$.  Similarly, $P(z)$ divides $z^{k-1} Q(z)Q(z^k)$ and since $P(0)=1$, and $P(z)$ and $Q(z)$ are relatively prime, we see that $P(z)$ divides $Q(z^k)$.  So we may write $Q(z^k)=P(z)b(z)$ with $b(z)$ dividing $(1-z)$.  Since $Q(0)=P(0)=1$, we see that $b(z)=1$ or $b(z)=(1-z)$.

Then substituting $P(z)=Q(z^k)/b(z)$ into Equation \eqref{PQ}, we find
$$z^{k-1}Q(z)b(z)b(z^k) = (1-z+z^{k-1}) Q(z^k)b(z^k) - (1-z) Q(z^{k^2}).$$
Now let $D$ denote the degree of $Q(z)$.  Then since \begin{equation}\label{Qkb}(1-z)Q(z^{k^2}) = - z^{k-1}Q(z)b(z)b(z^k) + (1-z+z^{k-1}) Q(z^k)b(z^k),\end{equation} and since $b(z)$ has degree at most $1$, we have
$$k^2D + 1 \leqslant \max\{ 2k+D, 2k-1+kD\}$$ so that \begin{equation}\label{Dkk1} D\leqslant \max\left\{ \frac{2k-1}{k^2-1},\frac{2}{k}\right\}\leqslant 1,\end{equation} since $k\geqslant 2$. Thus $D=0$ or $D=1$.

Suppose that $D=0$. Then $Q(z)$ is a constant polynomial and the condition that $Q(0)=1$ gives $Q(z)=1$. Since $P(z)$ divides $Q(z^k)$ we have that $P(z)$ is also $1$ and so $a(z)=z^{1-k}$.  But $$\frac{F_0(z^k)}{F_0(z)}=z^{1-k}(1-z+O(z^2)) \neq z^{1-k}=a(z),$$ and so we get a contradiction, thus $D=1$.

So, suppose that $D=1$. By \eqref{Dkk1} it is clear that if $k\geqslant 3$, then $D=0$, so we must have $k=2$.  If $b(z)=1$, then comparing degrees of the sides of the equality in \eqref{Qkb} gives $k^2+1=2k-1$, which is impossible since $k\geqslant 2$. Thus we must have $b(z)=1-z$. In this case, $Q(z)$ has degree one and we have $Q(z^2)=P(z)(1-z)$. Plugging in $z=1$ gives $Q(1)=0$ and since $Q(z)$ has degree $1$, we have $Q(z)=1-z$. Then $Q(z^2)=P(z)(1-z)$ gives that $P(z)=1+z$ and so $$a(z)=\frac{1}{z} \cdot\frac{1+z}{1-z} = \frac{1}{z}+2 +O(z).$$  But $$\frac{F(z^2)}{F(z)}=\frac{1}{z}- 1+O(z),$$ and so we obtain a contradiction.  Thus $F_0(z)$ and $F_0(z^k)$ are linearly independent over $\mathbb{C}(z)$.
\end{proof}

\begin{lemma}\label{Xnotzero} Let $F_0(z)$ be as defined above, let $r\in\mathbb{N}$, and let $h_0(z),\ldots,h_r(z)$ be rational functions such that $h_i(z)/z^{k^i-1}$ does not have a pole at $z=0$ for $i=0,\ldots ,r$.  Then, if $$\sum_{i=0}^r h_i(z) F_0(z^{k^i})= 0,$$ then $h_0(0)=0$.
\end{lemma}

\begin{proof} We prove this by induction on $r$. For $r=0$ and $r=1$, the result follows by Lemma \ref{independent} since $F_0(z)$ and $F_0(z^k)$ are linearly independent over $\mathbb{C}(z)$.
So suppose that the result holds for $r<m$ with $m\geqslant 2$ and consider the case when $r=m$.

Towards a contradiction, suppose that $$\sum_{i=0}^m h_i(z) F_0(z^{k^i})  = 0$$ with $h_0(0)$ nonzero and $z^{k^i-1}$ dividing $h_i(z)$ in the local ring $\mathbb{C}[z]_{(z)}$ (recall that $\mathbb{C}[z]_{(z)}$ is the ring of all rational functions whose denominator, when written in reduced form, is nonzero at $z=0$).  For $i=1,\ldots,m$, set $$g_i(z) := z^{-k^i+1} \frac{h_i(z)}{h_0(z)}.$$ Then since $h_0(0)$ is nonzero, each $g_i(z)$ is regular at $z=0$ and
$$F_0(z) + \sum_{i=1}^m g_i(z) z^{k^i-1} F_0(z^{k^i}) = 0.$$
Applying the Cartier operator $\Lambda_0$ gives
$$\Lambda_0(F_0)(z) + \sum_{i=1}^m \Lambda_0( g_i(z)z^{k-1}) z^{k^{i-1}-1} F_0(z^{k^{i-1}}) = 0.$$ But using \eqref{F_1} and applying Lemma \ref{ch6:Cartier}(a), we have $\Lambda_0(F_0)(z) = F_0(z) - z^{k-1} F_0(z^k)$, so we have
\begin{multline}\label{lambdaF_1} 0 =  (1+\Lambda_0(g_1(z)z^{k-1})) F_0(z)\\ +
 (-1+\Lambda_0(g_2(z) z^{k-1})) z^{k-1} F_0(z^k)   \\ + \sum_{i=2}^{m-1} \Lambda_0(g_{i+1}(z)z^{k-1})  z^{k^{i}-1}F_0(z^{k^{i}}).  \end{multline}
Since $g_1(z)$ is regular at $z=0$, we have that $g_1(z)z^{k-1}$ has a power series expansion with zero constant term and hence $\Lambda_0(g_1(z)z^{k-1})$ vanishes at $z=0$, and so $1+\Lambda_0(g_1(z)z^{k-1})$ is a rational function which is nonzero at $z=0$. Since each of the higher-index coefficients in \eqref{lambdaF_1} are of the form $z^{k^{i}-1}$ times a rational function regular at $z=0$, the induction hypothesis applies and we get a contradiction. This contradiction proves the lemma.
\end{proof}

\begin{proof}[Proof of Theorem \ref{thm: example}] Let $F(z)$ be the $k$-regular power series defined in \eqref{eqF}.  We claim that there is no nonzero rational function $R(z)$ such that function $R(z)F(z)$ is a $k$-Becker power series.  Since $F(0)=1$, if $R(z)F(z)$ has a power series expansion at $z=0$, $R(z)$ must be regular at $z=0$.
Suppose towards a contradiction that there is a rational function $R(z)$ such that $R(z)$ is regular at $z=0$ and such that $F(z)R(z)$ is $k$-Becker.  Then we write $R(z)=z^a R_0(z)$ with $a\geqslant 0$ and with $R_0(0)$ nonzero.  Then there exist a natural number $d$ and polynomials $b_1(z),\ldots,b_d(z)$ such that
$$R_0(z)F(z) = b_1(z) z^{ka-a} R_0(z^k) F(z^k) + \cdots + b_d(z) z^{k^d a - a} R_0(z^{k^d}) F(z^{k^d}).$$ As defined above, $F(z)=z F_0(z)$, so we have
\begin{align}\label{star}R_0(z) F_0(z) &= b_1(z) z^{ka-a+ k-1} R_0(z^k)  F_0(z^k)\\ \nonumber &\qquad\qquad+ \cdots + b_d(z) z^{k^da -a + k^d-1} R_0(z^{k^d}) F_0(z^{k^d}). \end{align}
But this contradicts Lemma \ref{Xnotzero}.  The result follows.
\end{proof}

\section{A structure of Mahler functional equations for regular functions}\label{structure}

In this section, we prove Proposition \ref{prop:dumas}; that is, we show for $F(z)\in\mathbb{C}[[z]]$, the series $F(z)$ is $k$-regular if and only if $F(z)$ satisfies some functional equation \eqref{MFE} such that all of the zeros of $a_0(z)$ are either zero or roots of unity of order not coprime to $k$. As stated in the Introduction, Proposition \ref{prop:dumas} is obtained by combining Theorem \ref{main} with a result of Dumas \cite[Th\'eor\`eme 30]{Dumasthese}. Dumas's result \cite[Th\'eor\`eme 30]{Dumasthese} is proved by appealing to results for degree-one Mahler functions via his Structure Theorem recorded above as Theorem \ref{lem:DAB}. By appealing to Theorem \ref{lem:DAB} and the ring structure of the set of $k$-regular power series, one can show that a series $F(z)$ is $k$-regular, if one can show that the infinite product $$H(z):=\frac{1}{\prod_{j\geqslant 0}\Gamma(z^{k^j})}$$ is $k$-regular. This is exactly what Dumas did via the following lemma; see \cite[Lemme~8]{Dumasthese}.

\begin{lemma}[Dumas]\label{lem:partial} The infinite product $H(z)=\prod_{j\geqslant 0}\Gamma(z^{k^j})^{-1}$ is $k$-regular if and only if the $\mathbb{C}$-vector space $$\left\langle\left\{\Lambda_{r_n}\cdots \Lambda_{r_1}\left(\frac{1}{\prod_{j=0}^{n-1}\Gamma(z^{k^j})}\right): 0\leqslant r_i<k,\ n\in\mathbb{N}\right\}\right\rangle_\mathbb{C}$$ is finite-dimensional.
\end{lemma}

Lemma \ref{lem:partial} follows from Lemma \ref{cartvec} combined with the equality $$\Lambda_{r_n}\cdots \Lambda_{r_1}H(z)=\left(\Lambda_{r_n}\cdots \Lambda_{r_1}\left(\frac{1}{\prod_{j=0}^{n-1}\Gamma(z^{k^j})}\right)\right) H(z),$$ which itself follows from the fact that $H(z)$ is a degree-one Mahler function satisfying the functional equation $$\Gamma(z)H(z)-H(z^k)=0.$$

We require the following proposition for the necessary direction of Proposition \ref{prop:dumas}. As stated previously, the argument is due to Dumas \cite[Th\'eor\`eme 30]{Dumasthese}. We state the result here in a slightly different form.

\begin{proposition}[Dumas]\label{regprod} Let $\Gamma(z)$ be a polynomial with $\Gamma(0)=1$. If all of the zeros of $\Gamma(z)$ are roots of unity of order not coprime to $k$, then $H(z)=\prod_{j\geqslant 0}\Gamma(z^{k^j})^{-1}$ is $k$-regular.
\end{proposition}

\noindent To prove Proposition \ref{regprod}, Dumas proved that the functions $$\Lambda_{r_n}\cdots \Lambda_{r_1}\left(\prod_{j=0}^{n-1}\Gamma(z^{k^j})^{-1}\right),$$ for $n\geqslant 1$, have only finitely many poles with bounded multiplicities and then applied Lemma~\ref{lem:partial}; see also \cite[Theorem 10]{D1993}. Compare with Lemma \ref{Buniform}, where we show a similar result for the set of matrices $\{{\bf B}_n(z):n\geqslant 1\}$.

For the sufficient direction of Proposition \ref{prop:dumas}, we will use the following result.

\begin{lemma}\label{Qzk} Let $k\geqslant 2$ be an integer, $Q(z)$ be a polynomial and suppose that all of the zeros of $Q(z)$ are either zero or roots of unity of order not coprime to $k$. Then for any integer $m\geqslant 1$, the zeros of $Q(z^{k^m})$ are either zero or roots of unity of order not coprime to $k$.
\end{lemma}

\begin{proof} Since all zeros of $Q(z)$ are either zero or roots of unity, it is clear that all zeros of $Q(z^{k^m})$ are either zero or roots of unity.

Now suppose to the contrary that there is a zero $z=\zeta$ of $Q(z^{k^m})$ that is a root of unity of order coprime to $k$, say $\ell$. Then since $\gcd(k,\ell)=1$, there is a positive integer $M$ dividing $\varphi(\ell)$ such that $k^M\equiv 1\ (\bmod\ \ell)$. Thus for this $M$, we have \begin{equation}\label{zz}\zeta^{k^M}=\zeta.\end{equation} Since $z=\zeta$ is a zero of $Q(z^{k^m})$, we have that $z=\xi:=\zeta^{k^m}$ is a zero of $Q(z)$. But then, using \eqref{zz}, we have $z=\xi$ is a zero of $Q(z)$ such that $$\xi=\zeta^{k^{m}}=\left(\zeta^{k^{M}}\right){}^{k^m}=\left(\zeta^{k^{m}}\right){}^{k^M}={\xi}^{k^M}.$$ If we denote by $n$ the order of $\xi$, this gives that $k^M\equiv 1\ (\bmod\ n)$, so that we have $\gcd(k,n)=1$, a contradiction, which proves the lemma.
\end{proof}

\begin{proof}[Proof of Proposition \ref{prop:dumas}] We prove sufficiency first. Towards this, suppose that $F(z)$ is $k$-regular and satisfies the minimal functional equation \eqref{MFE}. Following the comments after Theorem \ref{main}, we denote by $A$ the set of roots of unity $\zeta$ such that $\zeta^{k^M}\neq\zeta$ for all $M\geqslant 1$ and $a_0(\zeta)=0$; note that this condition is equivalent to the condition that the order of $\zeta$ is not coprime to $k$. Then there is a nonnegative integer $\gamma$ and an $N$ depending on $a_0(z)$ such that for $$Q(z):=\prod_{\zeta\in A}\prod_{j=0}^{N-1}(1-z^{k^j}\overline{\zeta}^{k^N})^{\nu_\zeta(a_0)},$$ the function $F(z)/z^\gamma Q(z)$ satisfies a Mahler-type functional equation \eqref{MFE} with $a_0(z)=1$. In particular, we write $$\frac{F(z)}{z^\gamma Q(z)}+\sum_{i=1}^D b_i(z)\cdot \frac{F(z^{k^i})}{z^{\gamma k^i} Q(z^{k^i})}=0.$$ Now multiplying by $z^{\gamma k^D} Q(z) Q(z^k)\cdots Q(z^{k^d})$ gives \begin{multline}\label{newMFE} \qquad z^{\gamma (k^{D}-1)}Q(z^k)\cdots Q(z^{k^D}) F(z)\\ +\sum_{i=1}^D b_i(z)z^{\gamma (k^D-k^i)}\left(\frac{\prod_{j=0}^D Q(z^{k^j})}{Q(z^{k^i})}\right) \cdot F(z^{k^i})=0.\qquad \end{multline} By the definition of $Q(z)$ and Lemma \ref{Qzk}, we have that $F(z)$ satisfies a (new) functional equation \eqref{MFE}, specifically Equation \eqref{newMFE}, such that all of the zeros of $$a_0(z)=z^{\gamma (k^{D}-1)}Q(z^k)\cdots Q(z^{k^D})$$ are either zero or roots of unity of order not coprime to $k$. This proves necessity.

For sufficiency, we use both Theorem \ref{lem:DAB} and Proposition \ref{regprod}. To this end, suppose that $F(z)$ satisfies some functional equation \eqref{MFE} such that all of the zeros of $a_0(z)$ are either zero or roots of unity of order not coprime to $k$. Now write $$a_0(z)=\rho z^\delta \Gamma(z),$$ where $\Gamma(0)=1$. Thus all the zeros of $\Gamma(z)$ are roots of unity of order not coprime to $k$. Now, Theorem \ref{lem:DAB}, gives that there is a $k$-regular series $G(z)$ such that $$F(z)=\frac{G(z)}{\prod_{j\geqslant 0}\Gamma(z^{k^j})}.$$ Applying Proposition \ref{regprod} gives that the function $$H(z):= \frac{1}{\prod_{j\geqslant 0}\Gamma(z^{k^j})}$$ is $k$-regular. Since $k$-regular series form a ring, we have that $F(z)=G(z)H(z)$ is $k$-regular. This proves sufficiency, and completes the proof of the proposition.
\end{proof}

\bibliographystyle{amsplain}
\providecommand{\bysame}{\leavevmode\hbox to3em{\hrulefill}\thinspace}
\providecommand{\MR}{\relax\ifhmode\unskip\space\fi MR }
\providecommand{\MRhref}[2]{%
  \href{http://www.ams.org/mathscinet-getitem?mr=#1}{#2}
}
\providecommand{\href}[2]{#2}


\end{document}